\documentclass[a4paper,12pt]{article}

\usepackage{amssymb}
\usepackage{latexsym}
\usepackage{amsfonts}
\usepackage{amsmath}
\usepackage{amsthm}
\usepackage{hyperref}
\usepackage{tensor}
\usepackage{mathtools}
\usepackage{mdwlist}
\usepackage{pgf,tikz,pgfplots}
\pgfplotsset{width=7.5cm,compat=1.9}

\usepackage[shortlabels]{enumitem}

\usetikzlibrary{arrows,patterns,positioning}
\usetikzlibrary{calc}

\usepackage[margin=1in]{geometry}

\usepackage[compact]{titlesec}

 \usepackage[scale=0.9]{tgheros}
\usepackage[T1]{fontenc}

\usepackage{soul}

\DeclareMathOperator{\Aut}{Aut}

\DeclareMathOperator{\PSL}{PSL}
\DeclareMathOperator{\SL}{SL}

\DeclareMathOperator{\pg}{PG}
\DeclareMathOperator{\ag}{AG}

\DeclareMathOperator{\Sp}{Sp}

\DeclareMathOperator{\GO}{GO}

\DeclareMathOperator{\PGaL}{P\Gamma L}

\DeclareMathOperator{\AGL}{AGL}

\DeclareMathOperator{\GL}{GL}
\DeclareMathOperator{\GaL}{\Gamma L}

\DeclareMathOperator{\SU}{SU}

\renewcommand{\leq}{\leqslant}
\renewcommand{\geq}{\geqslant}
\newcommand{\ra}{\rightarrow}
\newcommand{\sset}{\subseteq}

\newcommand{\F}{\mathbb F}

\newcommand{\D}{\mathcal D}
\newcommand{\B}{\mathcal B}
\renewcommand{\P}{\mathcal P}

\usepackage{stackengine,scalerel}
\stackMath

\theoremstyle{plain}
\newtheorem{lemma}{Lemma}
\newtheorem{theorem}[lemma]{Theorem}

\theoremstyle{definition}
\newtheorem{definition}[lemma]{Definition}

\newtheorem{remark}[lemma]{Remark}

\numberwithin{equation}{section}
\numberwithin{lemma}{section}

\title{The Non-Existence of Block-Transitive Subspace Designs\thanks{The first author is supported by the Croatian Science Foundation under the project 6732. The second author acknowledges the support of the Australian Research Council Discovery Grant DP200101951. This work was supported by resources provided by the Pawsey Supercomputing Centre with funding from the Australian Government and the Government of Western Australia.}}
\author{
 Daniel R. Hawtin$^1$
 \and
 Jesse Lansdown$^2$
}

\date{
 \small{
  \emph{
   $^1$Department of Mathematics, University of Rijeka\\
   Rijeka, Croatia, 51000.\\
  }
  \href{mailto:dan.hawtin@gmail.com}{dan.hawtin@gmail.com}\\
  \vspace{0.25cm}
 }
 \small{
  \emph{
   $^2$Centre for the Mathematics of Symmetry and Computation,\\
   The University of Western Australia\\
   Perth, Western Australia, 6009.\\
  }
  \href{mailto:jesse.lansdown@uwa.edu.au}{jesse.lansdown@uwa.edu.au}\\
  \vspace{0.25cm}
 }
 \today
}
\begin{document}
\maketitle

\begin{abstract}
 Let $q$ be a prime power and $V\cong\F_q^d$. A \emph{$t$-$(d,k,\lambda)_q$ design}, or simply a \emph{subspace design}, is a pair $\D=(V,\B)$, where $\B$ is a subset of the set of all $k$-dimensional subspaces of $V$, with the property that each $t$-dimensional subspace of $V$ is contained in precisely $\lambda$ elements of $\B$. Subspace designs are the \emph{$q$-analogues} of balanced incomplete block designs. Such a design is called \emph{block-transitive} if its automorphism group $\Aut(\D)$ acts transitively on $\B$. It is shown here that if $t\geq 2$ and $\D$ is a block-transitive $t$-$(d,k,\lambda)_q$ design then $\D$ is trivial, that is, $\B$ is the set of all $k$-dimensional subspaces of $V$. 
\end{abstract}

\section{Introduction and preliminaries}

Tits \cite{tits1957analogues} suggested that combinatorics of sets could be regarded as the limiting case $q\ra 1$ of combinatorics of vector spaces over the finite field $\F_q$. Taking a combinatorial property expressed in terms of sets and rephrasing its definition in terms of $\F_q$-vector spaces gives rise to what has become known as the \emph{$q$-analogue} of the original property. A \emph{$t$-$(d,k,\lambda)$ design} is a pair $(\P,\B)$ where $\P$ is a set of size $d$ and the elements of $\B$ are $k$-subsets of $\P$, called \emph{blocks}, satisfying the condition that every $t$-subset of $\P$ is contained in precisely $\lambda$ blocks.
The $q$-analogue of a $t$-$(d,k,\lambda)$ design is a $t$-$(d,k,\lambda)_q$ design. A precise definition of $t$-$(d,k,\lambda)_q$ designs is given in Definition~\ref{subspacedesdef}.
For brevity we often refer to $t$-$(d,k,\lambda)$ designs and $t$-$(d,k,\lambda)_q$ designs simply as \emph{block designs} and \emph{subspace designs}, respectively.
Subspace designs were first referenced in the literature by Cameron \cite{cameron1973generalisation}. See the recent survey of Braun et al.~\cite{braun2018q} for more in-depth background on subspace designs. 
A block or subspace design is \emph{block-transitive} if it admits a group of automorphisms that acts transitively on its set of blocks (see Definition~\ref{autgroupdef} for a precise definition of the automorphism group of a subspace design).

Automorphism groups of block designs have been studied since the mid-twentieth century (see, for example, \cite{hall1960automorphisms,parker1957collineations}). There are many examples of block designs having large automorphism groups and restricting attention to a subclass of block designs with a high degree of symmetry can be a useful tool for studying designs (for example \cite{Kantor1969,kantor19752}); classification of such a subfamily may even be possible (for example \cite{kantor1985classification}). Block-transitive block designs have been studied since the 1980s (see \cite{camina1984block,camina1989block}) and are a vast enough class of designs that authors often restrict their study to designs satisfying additional conditions (see \cite{cameron1993block,huiling1995block,o1993block}). 

In contrast to the situation with block designs, there are no known examples of block-transitive subspace designs. Moreover, the known subspace designs do not have particularly large automorphism groups. For instance, the automorphism groups of the subspace designs in \cite{braunqSteinerexist2016,miyakawa1995class,suzuki19902,suzuki19922,thomas1987designs} are all normalisers of Singer cycles. Furthermore, if a binary $q$-analogue of the Fano plane exists then its automorphism group has size at most $2$ \cite{kiermaier2018order}. Indeed, our main result, Theorem~\ref{maintheorem}, shows that non-trivial block-transitive subspace designs do not exist.

\begin{theorem}\label{maintheorem}
There exist no non-trivial block-transitive $t$-$(d,k,\lambda)_q$ designs for $2\leq t<k<d$ and $q$ a prime power.
\end{theorem}


Theorem \ref{maintheorem} also rules out the existence of subspace designs having stronger forms of symmetry that are often studied for block designs, such flag-transitivity, that is, transitivity on incident point-block pairs. However, most of the known subspace designs (in particular, those mentioned above that are invariant under a Singer cycle) do satisfy the weaker symmetry condition that they are point-transitive, that is, their automorphism groups act transitively on the $1$-dimensional subspaces of the underlying vector space.

The proof of Theorem~\ref{maintheorem} relies on work of Bamberg and Penttila \cite[Theorem~3.1]{bamberg2008overgroups}, who determined all linear groups having orders with certain divisors. Their result in turn relies on Guralnick et al.~\cite{guralnick1999linear}, who make use of the Aschbacher classification of maximal subgroups of classical groups \cite{aschbacher1984maximal} and the classification of finite simple groups. We consider in Section~\ref{mainproof} each of the cases determined by \cite[Theorem~3.1]{bamberg2008overgroups} in order to obtain Theorem~\ref{maintheorem}. The majority of cases involve a simple analysis, with the exception of those treated separately in Section~\ref{smallcases}. In particular, the case treated in Lemma~\ref{jesseslemma} involves an exhaustive computer search.

In Section~\ref{subdesignprelim} we introduce notation and preliminary results for subspace designs. In Section~\ref{primitivesect} we discuss the concept of a primitive divisor and set up the application of \cite{bamberg2008overgroups}. Section~\ref{smallcases} deals with several specific cases. Finally, in Section~\ref{mainproof} we prove Theorem~\ref{maintheorem}.

\subsection{Subspace designs}\label{subdesignprelim}

In analogy with the binomial coefficient ${n \choose k}$ we define the \emph{$q$-binomial coefficient},
\[
 {d \choose k}_q=\frac{(q^d-1)\cdots (q^{d-k+1}-1)}{(q^k-1)\cdots(q-1)}.
\]
The $q$-binomial coefficient is also sometimes referred to as the \emph{Gaussian coefficient}.
Similarly, this time in analogy with the set ${N \choose k}$ of all $k$-subsets of a set $N$, we denote
the set of all $k$-dimensional subspaces of a vector space $V$ over $\F_q$ by ${V \choose k}_q$.

\begin{definition}\label{subspacedesdef}
 Given integers $d,k,t$, and $\lambda$, with $1\leq t<k\leq d-1$, a \emph{$t$-$(d,k,\lambda)_q$ design} (or briefly a \emph{$q$-design} or \emph{subspace design}) is a pair $\D=(V,\B)$, where $V=\F_q^d$ and $\B\sset {V \choose k}_q$, such that each element of ${V \choose t}_q$ is a subspace of precisely $\lambda$ elements of $\B$.
\end{definition}

 A \emph{$q$-Steiner system} is a $t$-$(d,k,\lambda)_q$ design with $\lambda=1$. A subspace design with $t=1$ is known as a \emph{$k$-covering} or if additionally $\lambda=1$ a \emph{$k$-spread}. Since $k$-coverings and $k$-spreads have been studied in their own right we always assume here that $t\geq 2$. Note that for a (classical) $t$-$(d,k,\lambda)$ design the case $t \geq 2$ also attracts the most interest. The $t$-$(d,k,\lambda)_q$ design with $\B= {V \choose k}_q$ is the \emph{trivial} design. The number of blocks in a $t$-$(d,k,\lambda)_q$ design $\D=(V,\B)$ is given by
 \[
  |\B|= \lambda  \frac{{d \choose t}_q }{ {k \choose t}_q }=\lambda\frac{(q^d-1) \cdots (q^{d-t+1}-1)}{(q^k-1) \cdots(q^{k-t+1}-1)}.
 \]
 
 Given a non-singular sesquilinear form $f$ defined on $V$ we obtain the \emph{dual design} $\D^\perp=(V,\B^\perp)$ of a subspace design $\D=(V,\B)$, where $\B^\perp=\{U^\perp\mid U\in \B\}$ and for each $U\leq V$ we have $U^\perp=\{v\in V\mid f(u,v)=0, \forall u\in U\}$. For more details on duality, see \cite[Section~2.1]{braun2018q}. The dual of a subspace design is indeed again a subspace design by \cite[Lemma~4.2]{suzuki1990inequalities}. A subspace design such that $\D=\D^\perp$ is called \emph{self-dual}.

Let $V\cong \F_q^d$. The \emph{Grassmann graph} $J_q(d,k)$ is the graph having vertex set ${V \choose k}_q$, where two vertices are adjacent precisely when they intersect in a $(k-1)$-dimensional subspace (see \cite[Section~9.3]{brouwer}). The automorphism group of $J_q(d,k)$ satisfies
\begin{itemize}
 \item $\Aut(J_q(d,k))\cong \PGaL_d(q)$ when $1<k<d$ and $2k\neq d$,
 \item $\Aut(J_q(d,k))\cong \PGaL_d(q) . C_2$ when $2k=d$ (cf. \cite[Theorem 9.3.1]{brouwer}).
\end{itemize}
If $\D=(V,\B)$ is a $t$-$(d,k,\lambda)_q$ design, then $\B$ is a subset of the vertex set of $J_q(d,k)$, which leads us to state Definition~\ref{autgroupdef} in its present form. Note that this differs from the convention used by some authors, where the automorphism group of $\D$ is defined to be a subgroup of the automorphism group of the subspace lattice, that is, a subgroup of $\PGaL_d(q)$ (see the discussion in \cite[Section~2.1]{braun2018q}). In particular, viewing $\Aut(\D)$ as a subgroup of $\Aut(J_q(d,k))$ allows us to consider outer automorphisms of $\PGaL_d(q)$ (anti-automorphisms of the subspace lattice) to act as automorphisms of $\D$.  Our approach is equivalent to considering the relationships of subspaces via incidence (symmetrised inclusion), rather than inclusion; this being the weakest possible structure to impose on the set of subspaces without losing the overall subspace structure. It follows that there can be no further automorphisms than those in $\Aut(J_q(d,k))$ that preserve the desired structure. Note that Definition~\ref{autgroupdef} implies that when $\D$ is not self-dual, in particular if $d\neq 2k$, then we may assume the automorphism group of $\D$ is a subgroup of $\PGaL_d(q)$, whilst if $d=2k$ we must consider the larger group $\PGaL_d(q) . C_2$, where the group $C_2$ is generated by a duality.

\begin{definition}\label{autgroupdef}
 Let $\D=(V,\B)$ be a $t$-$(d,k,\lambda)_q$ design. The automorphism group $\Aut(\D)$ of $\D$ is defined to be the setwise stabiliser of $\B$ inside the automorphism group of the Grassmann graph $J_q(d,k)$. Moreover, $\D$ is called \emph{block-transitive} if $\Aut(\D)$ acts transitively on $\B$.
\end{definition}

The following result shows that the dual of a block-transitive design is again block-transitive and allows us to restrict our attention to the case where $k\leq d/2$ in Section~\ref{mainproof}. 

\begin{lemma}\label{dualdesignlemma}
 Let $f$ be a non-singular sesquilinear form on $V$ and $\D=(V,\B)$ be a block-transitive $t$-$(d,k,\lambda)_q$ design. Then the dual $\D^\perp$ of $\D$ is a block-transitive $t$-$(d,d-k,\lambda')_q$ design, where $\lambda'=\lambda {d-t \choose k}_q / {d-t \choose k-t}_q$.
\end{lemma}

\begin{proof}
 By \cite[Lemma~4.2]{suzuki1990inequalities}, we have that $\D^\perp$ is a $t$-$(n,n-k,\lambda')_q$ design. Let $G=\Aut(\D)$. Now, the dual map $U\mapsto U^\perp$ induces a bijection between ${V\choose k}_q$ and ${V\choose {d-k}}_q$ and a corresponding automorphism $\tau$ of $\PGaL_d(q)$ which together define a permutational isomorphism between the action of $G$ on $\B$ and the conjugate group $G^\tau$ acting on $\B^\perp$. It thus automatically follows that $\D$ is block-transitive if and only if $\D^\perp$ is block-transitive.
\end{proof}

The next result is simply a special case of \cite[Lemma~2.1]{suzuki1990inequalities} and allows us to assume that $t=2$ in Section~\ref{mainproof}.

\begin{lemma}\label{twodesigncor}
 If $\D=(V,\B)$ is a $t$-$(d,k,\lambda)_q$ design, for some $t\geq 2$, then $\D$ is also a $2$-$(d,k,\lambda_2)_q$ design, where
\[
 \lambda_2=\lambda\frac{ {d-2\choose t-2}_q }{ {k-2\choose t-2}_q }
\]
is an integer, and
\[
 |\B|=\lambda\frac{ {d\choose t}_q }{ {k\choose t}_q }=\lambda_2\frac{ {d\choose 2}_q }{ {k\choose 2}_q }=\lambda_2\frac{(q^d-1)(q^{d-1}-1)}{(q^k-1)(q^{k-1}-1)}.
\]
\end{lemma}

\subsection{Primitive divisors and linear groups}\label{primitivesect}

A divisor $r$ of $q^e-1$ that is coprime to each $q^i-1$ for $i<e$ is said to be a \emph{primitive divisor}. The \emph{primitive part} of $q^e-1$ is the largest primitive divisor, and we denote the primitive part of $q^e-1$ by $\Phi_e^*(q)$. Note that, since $\Phi_1^*(q)=q-1$ is even when $q$ is odd and $q^e-1$ itself is odd if $q$ is even, we have that $\Phi_{e}^*(q)$ is always odd. We then have the following.

\begin{lemma}\label{divislemma}
 Let $\D=(V,\B)$ be a block-transitive $2$-$(d,k,\lambda)_q$ design with $k\leq d/2$ and let $G=\Aut(\D)\cap \PGaL_d(q)$. Then $\Phi_d^*(q)\cdot\Phi_{d-1}^*(q)$ divides the order of $G$.
\end{lemma}

\begin{proof}
 First, note that either $G$ is actually equal to $\Aut(\D)$, or $\D$ is self-dual and $G$ has index $2$ inside $\Aut(\D)$. Since $\Aut(\D)$ acts transitively on $\B$, it follows that $|\B|$ divides $|\Aut(\D)|$, and hence also divides $2|G|$. By definition we have that
 \[
  \gcd\left(\Phi_d^*(q),\Phi_{d-1}^*(q)\right)=1.
 \]
 Moreover, $\Phi_i^*(q)$ is always odd. Thus, it suffices for us to prove that $\Phi_i^*(q)$ divides $2|G|$ for each $i=d,d-1$. Now
 \[
  |\B|=\lambda\frac{(q^d-1)(q^{d-1}-1)}{(q^k-1)(q^{k-1}-1)}.
 \]
 Since $2<k$ and $k\leq d/2$, we have that $k<d-1$. Thus, for $i=d,d-1$, we have that $\Phi_i^*(q)$ is coprime to each of $q^k-1$ and $q^{k-1}-1$. Hence $\Phi_d^*(q)\Phi_{d-1}^*(q)$ divides $|\B|$, and therefore $|G|$, as required.
\end{proof}

The significance of Lemma~\ref{divislemma} is that it allows \cite[Theorem~3.1]{bamberg2008overgroups} to be applied in Section~\ref{mainproof}. Note that the original statement of \cite[Theorem~3.1]{bamberg2008overgroups} spans $3$ pages and includes tables of viable parameters. We provide an abridged version of the theorem, and refer the reader to \cite{bamberg2008overgroups} for finer details. Note that we write $V_d(q)\cong \F_q^d$.

\begin{remark}\label{extfielderror}
 Before stating the abridged \cite[Theorem~3.1]{bamberg2008overgroups}, we note that there is a very small error in the statement of the extension field case. In particular, case (a) of the extension field case does not require that $b$ be a non-trivial divisor of $\gcd(d,e)$; see \cite[Example~2.4]{guralnick1999linear} for clarification regarding the conditions in this case.
\end{remark}

\begin{theorem}[{\cite[Theorem 3.1]{bamberg2008overgroups}}]\label{theoremBP}
Let $q=p^f$ where $p$ is a prime, let $d$ and $e$ be integers greater than $2$ and satisfying $d/2 < e \leqslant d$. If a subgroup $G$ of $GL_d(q)$ has order divisible by $\Phi^*_{ef}(p)$, and $\Phi^*_{ef}(p)>1$, then one of the following occurs.
\begin{itemize}
	\item[] {\bf Classical Examples:} $G$ preserves a nondegenerate sesquilinear form on $V_d(q)$ and one of the following holds: (a) $\SL_d(q) \trianglelefteq G$; (b) $\Sp_d(q) \trianglelefteq G$; (c) $q$ is a square, $\SU_d(q) \trianglelefteq G$, and $e$ is odd; (d) $\Omega_d^\epsilon (q) \trianglelefteq G$ where $\epsilon = \pm$ for $d$ even, and $\epsilon = \circ$ when $dq$ is odd.
	\item[] {\bf Reducible Examples:} $G$ fixes a subspace or quotient space $U$ of $V_d(q)$ and $\dim(U)=m\geqslant e$. So $G \leqslant q^{m(d-m)} \cdot (\GL_m(q)\times GL_{d-m}(q))$ and $\Phi^*_{ef}(p)$ divides $|G^U|$.
	\item[] {\bf Imprimitive Examples:} Here $q=p$, $\Phi^*_e(p) = e+1$, and $G$ preserves  a direct sum decomposition $V= U_1 \oplus \cdots \oplus U_d$, where each $U_i$ has dimension $1$. Moreover, $G \leqslant \GL_1(q) \wr S_d$ in product action, and $G$ induces a primitive group on the factors $\{U_1, \ldots, U_d\}$. There are a finite number of possible values that $q$, $e$, and $d$ take.
	\item[] {\bf Extension Field Examples:} Here we have that there is a non-trivial divisor $b$ of $d$, such that $G$ preserves on $V_d(q)$ a field extension structure of a vector space $V_{d/b}(q^b)$. Therefore $G \leqslant \GaL_{d/b}(q^b)$. Two subcases occur, according to whether $\Phi^*_{ef}(p)$ is coprime to $b$ or not. 
	\item[] {\bf Symplectic Type Examples:} Here $q=p$, $\Phi^*_e(p) = e+1$, and $G$ normalises an extraspecial $2$-group. Specifically, we have one of the following: (a) $p=3$, $e=d=4$, and $G \leqslant (2_{-}^{1+4} \cdot O_4^-(2)) \circ 2$; (b) $p=3$, $d=8$, $e=6$ and $G \leqslant (2_{+}^{1+6} \cdot O_6^+(2)) \circ 2$; (c) $p=5$, $d=8$, $e=6$ and either $G \leqslant ((4 \circ 2^{1+6}) \cdot \Sp_6(2))\circ 4)$ or $G \leqslant (2_+^{1+6} \cdot O_6^+(2)) \circ 4$.
	\item[] {\bf Nearly Simple Examples:} In this case, $S \leqslant \overline{G} \leqslant \Aut(S)$, where $S$ is a finite nonabelian simple group. The following subcases occur: (a) the alternating group case; (b) the sporadic simple group case; (c) the cross-characteristic case; (d) the natural-characterisic case. In each subcase, possible parameters belong to a finite list of values.
\end{itemize}
\end{theorem}

\section{Special cases}\label{smallcases}

In the next section we deal with some special cases that either do not satisfy the conditions required in Section~\ref{mainproof} in order to apply Theorem \ref{theoremBP}, or otherwise require particular attention.

\begin{lemma}\label{orderofGsmallcase}
 Suppose $\D=(V,\B)$ is a block-transitive $2$-$(6,3,\lambda)_2$ design, for some $\lambda\geq 1$, and let $G\cong \Aut(\D)\cap\PGaL_6(2)$. Then both the size of $\B$ and the order of $G$ are divisible by $93$.
\end{lemma}

\begin{proof}
 Note that if $\D$ is self-dual, then $G$ has index $2$ in $\Aut(\D)$, and otherwise $G=\Aut(\D)$. Since $\D$ is block-transitive, it follows that $|\B|$ divides $|\Aut(\D)|$. Hence, $|\B|$ divides $2|G|$. Now,
 \[
  |\B|=\frac{63\cdot 31}{7\cdot 3}\lambda=3\cdot 31\cdot \lambda.
 \]
 Thus $3\cdot 31=93$ divides the order of $G$.
\end{proof}

\begin{lemma}\label{pointhypstab}
 Suppose $\D=(V,\B)$ is a block-transitive $2$-$(6,3,\lambda)_2$ design, for some $\lambda\geq 1$, and let $G\cong \Aut(\D)\cap\PGaL_6(2)$. Then $G$ does not stabilise a $5$-dimensional subspace of $V$. 
\end{lemma}

\begin{proof}
 Since $\Aut(\D)$ acts transitively on $\B$, and either $G=\Aut(\D)$ or $\D$ is self-dual and $G$ has index $2$ in $\Aut(\D)$, it follows that $\B$ is either a single $G$-orbit or the union of two equal-sized $G$-orbits. 
 Now, by Lemma~\ref{orderofGsmallcase}, $|\B|$ must be divisible by $93$. In the first case, the length of the single $G$-orbit must thus be divisible by $93$. In the latter case, the length of each $G$-orbit must be $|\B|/2$, and since $\gcd(93,2)=1$, it follows that the size of each $G$-orbit must also be divisible by $93$.
 
 Let $v_1,\ldots,v_6$ be a basis for $V$ and assume for a contradiction that $G$ stabilises $W=\langle v_2,\ldots,v_6\rangle$. Furthermore, let $H = N\rtimes K\cong\AGL_5(2)$ be the stabiliser of $W$ inside $\SL_6(2)\cong \PSL_6(2) \cong \PGaL_6(2)$. In particular, for each $w\in W$ we obtain an element $\eta_w\in N$ given by the linear transformation defined by $\eta_w:v_1\mapsto v_1+w$ and $\eta_w:v_i\mapsto v_i$ for each $i\neq 1$. Moreover, $K$ is given by the natural action of $\SL_5(2)$ on $W$ extended to all of $V$. Note that $K$ acts on $N$, and that action is transitive on the non-identity elements of $N$.

 Consider the projection $P$ of $G$ into $H/N\cong \SL_5(2)$ via the homomorphism $n\sigma\mapsto \sigma$ for $n\in N$ and $\sigma\in K$.  Since $93$ is coprime to the order of $N$, it follows that $93$ divides the order of $P$. By \cite[Table~8.24]{classiccolva}, there are no maximal subgroups of $\SL_5(2)$ having order divisible by $93$, and hence also no proper subgroups of $\SL_5(2)$ having order divisible by $93$. Therefore, $P\cong \SL_5(2)$. Since the Schur multiplier of $\SL_5(2)$ is trivial and there is only one conjugacy class of groups isomorphic to $\SL_5(2)$ inside $H$, we may assume that $G$ contains $K$. If $G\cap N$ is the trivial group then $G=K$. If $G\cap N$ is non-trivial then, since the action of $K$ is transitive on the non-identity elements of $N$, it follows that $G\cap N=N$ and $G=H$. We now consider the $G$-orbits on ${V \choose 3}_2$ for each possibility of $G$.
 
 Suppose $G=K$. Then ${W \choose 3}_2$ is one $G$-orbit. Note that if we consider $V$ to be the underlying vector space of $\pg_5(2)$ then $W$ is the underlying vector space of a subgeometry $\pg_4(2)$ and the points not in this subgeometry can be considered to form an affine space $\ag_5(2)$, with $N$ acting as the group of translations of this affine space. A $3$-dimensional subspace of $V$ that is not in ${W \choose 3}_2$ decomposes into a $2$-dimensional subspace $X$ of $W$ and a $2$-flat of $\ag_5(2)$ lying in the parallel class corresponding to $X$. Since $K$ is also transitive on ${W \choose 2}_2$, it follows that $K$ has $2^5/2^2=8$ orbits on the $2$-flats of $\ag_5(2)$, corresponding to a further $8$ orbits on ${V \choose 3}_2$. Each orbit has length has length $155$, which is not divisible by $93$, giving a contradiction.

 Now suppose $G=H$. Then there are just two orbits of $G$ on ${V \choose 3}_2$, the first being ${W \choose 3}_2$, with length $155$, and the second being the set of all $3$-dimensional subspaces that intersect $W$ in a $2$-dimensional subspace, with length $1240$.  Neither length is divisible by $93$, a contradiction.
\end{proof}

\begin{lemma}\label{nodesignnequals6}
 No non-trivial block-transitive $2$-$(d,k,\lambda)_q$ design exists with $q^d=2^6$. 
\end{lemma}

\begin{proof}
 Suppose, for a contradiction, that such a subspace design $\D=(V,\B)$ exists. By definition we have that $k\geq 3$ and by Lemma~\ref{dualdesignlemma} we may assume that $k\leq d/2$. Since $q^d=2^6$, so that $d\leq 6$, it follows that $k=3$, $q=2$ and $d=6$. Let $G=\Aut(\D)\cap \PSL_6(2)$. By Lemma~\ref{orderofGsmallcase}, we then have that $3\cdot 31$ divides the order of $G$. 
 
 Now, $\D$ non-trivial implies that $\B$ is not the set of all $3$-dimensional subspaces of $V$, from which it follows that $G$ must be a proper subgroup of $\PGaL_6(2)\cong\PSL_6(2)$ and is thus contained in some maximal subgroup of $\PSL_6(2)$. By \cite[Tables~8.24 and 8.25]{classiccolva}, there are two conjugacy classes of such maximal subgroups that have order divisible by $31$, a representative of each class being isomorphic to $2^5:\GL_5(2)$ as the stabiliser of a $1$-dimensional subspace or a $5$-dimensional subspace of $V$, respectively. By Lemma~\ref{pointhypstab} there is no block-transitive design such that $G$ stabilises a $5$-dimensional subspace.
Moreover, if $G$ were to stabilise a $1$-dimensional subspace, then $\Aut(\D^\perp)\cap \PSL_6(2)$ would stabilise a $5$-dimensional subspace. However, by Lemma~\ref{dualdesignlemma}, $\D^\perp$ has the same parameters as $\D$, and hence there is no such design where $G$ stabilises a $1$-dimensional subspace.
\end{proof}

\begin{lemma}\label{nminus1is6}
 No non-trivial block-transitive $2$-$(d,k,\lambda)_q$ design exists with $q^{d-1}=2^6$.
\end{lemma}

\begin{proof}
 Suppose that such a subspace design $\D=(V,\B)$ exists and let $G=\Aut(\D)$. By definition we have that $k\geq 3$ and by Lemma~\ref{dualdesignlemma} we may assume that $k\leq d/2$. Since $d\leq 7$ we have that $k=3$, $q=2$ and $d=7$. Note that $d$ odd implies $\D$ is not self-dual and $G\leq\PGaL_7(2)\cong\PSL_7(2)$. If $G\cong\PSL_7(2)$ then $G$ acts transitively on ${V\choose 3}_2$, which implies that $\B={V\choose 3}_2$. However this is not the case, as $\D$ is non-trivial, and hence we deduce that $G$ is a proper subgroup of $\PSL_7(2)$. It follows from this that $G$ is contained in some maximal subgroup of $\PSL_7(2)$. Now, Lemma~\ref{twodesigncor} implies that $|\B|=3\cdot 127\cdot\lambda$. Since $\D$ is block-transitive it follows that $|\B|$ must divide $|G|$. By \cite[Tables~8.35 and 8.36]{classiccolva}, the only maximal subgroup of $\PSL_7(2)$ that has order divisible by $127$ is the normaliser of a Singer cycle. However, this group does not have order divisible by $3$, giving a contradiction. 
\end{proof}

\begin{lemma}\label{jesseslemma}
 There does not exist a $2$-$(11,5,5)_2$ design having automorphism group isomorphic to $\GaL_1(2^{11})$.
\end{lemma}

\begin{proof}
 Note that the dimension of $V$ is odd here, which implies that $\D$ is not self-dual, and hence $\Aut(\D)\leq\PGaL_{11}(2)\cong\SL_{11}(2)$. By \cite[Table~8.70]{classiccolva}, there is a unique conjugacy class of subgroups isomorphic to $\GaL_1(2^{11})$ in $\SL_{11}(2)$. Thus, without loss of generality, we may construct $G \cong \PGaL_1(2^{11})$ as the normaliser of $\langle g \rangle$ in  $\PGaL_{11}(2)$, for a randomly found element $g \in \PGaL_{11}(2)$ with $|g|=2047$. If a $2$-$(11,5,5)_2$ design were to exist with automorphism group isomorphic to $G$, then it must be a single orbit of $G$ on $5$-spaces of $\mathbb{F}_{2}^{11}$. By computer, it was found that none of the $157607$ orbits of $G$ on $5$-spaces of $\mathbb{F}_2^{11}$ yield a $2$-$(11,5,5)_2$ design; see Remark~\ref{compremark} for more information.
\end{proof}

\begin{remark}\label{compremark}
 The computation required to prove Lemma \ref{jesseslemma} was performed in GAP \cite{GAP4} with the package FinInG \cite{fining}. Note that a $2$-$(11,5,5)_2$ design can equivalently be described as a set of projective $4$-spaces of $\pg_{10}(2)$ such that every projective $1$-space is contained in precisely $5$ elements. This formulation lends itself more naturally to construction in FinInG. Since there are too many (specifically, $3548836819$) projective $4$-spaces to reasonably fit in memory we instead constructed orbits of $G$ by finding suitably many distinct and unique representatives from each orbit. Representatives were determined uniquely by choosing them to be lexicographically least in their orbits. These representatives, as well as GAP code, are made available for ease of verification at \cite{BlockTransitiveSubspaceDesignsSupplement}.
\end{remark}

\section{Main results}\label{mainproof}

In this section we prove Theorem~\ref{maintheorem} by applying Theorem~\ref{theoremBP} (referring to \cite[Theorem~3.1]{bamberg2008overgroups} for finer details).
We frequently reference the cases, which are related to the Aschbacher classes as in \cite{aschbacher1984maximal}, in the manner that they are listed in \cite{bamberg2008overgroups}. These cases are organised according to the following categories: \emph{classical}, \emph{reducible}, \emph{imprimitive}, \emph{extension field}, \emph{symplectic}, and \emph{nearly simple}. For information regarding specific groups see, for instance, \cite{wilson2009finite}.

The next result splits the treatment of $2$-$(d,k,\lambda)_q$ designs into three cases.

\begin{lemma}\label{threecaseslemma}
 Let $q=p^f$, where $p$ is prime, let $\D=(V,\B)$ be a block-transitive $2$-$(d,k,\lambda)_q$ design, and let $G$ be the setwise stabiliser of $\B$ inside $\GaL_d(q)$. Then one of the following holds:
 \begin{enumerate}
  \item At least one of $\Phi_{df}^*(p)$ or $\Phi_{(d-1)f}^*(p)$ is trivial.
  \item Both of $\Phi_{df}^*(p)$ and $\Phi_{(d-1)f}^*(p)$ are non-trivial and divide the order of $G$ where $G\leq\GaL_1(q^d)$.
  \item Both of $\Phi_{df}^*(p)$ and $\Phi_{(d-1)f}^*(p)$ are non-trivial and divide the order of ${\hat G}=G\cap \GL_d(q)$, and ${\hat G}$ is as in one of the classical, reducible, imprimitive, symplectic, or nearly simple cases of Theorem~\ref{theoremBP}.
 \end{enumerate}
\end{lemma}

\begin{proof}
 If either $\Phi_{df}^*(p)=1$ or $\Phi_{(d-1)f}^*(p)=1$ then part 1 holds. Suppose both of $\Phi_{df}^*(p)$ and $\Phi_{(d-1)f}^*(p)$ are non-trivial. Let $H=G/\left(G\cap Z(\GL_d(q))\right)$. Note that if $\D$ is self-dual then $G$ has index $2$ inside $\Aut(\D)$, and hence it is possible that $G$, and thus also $H$, has $2$ equal-sized orbits on $\B$. Since $H$ is a quotient of $G$, it follows from Lemma~\ref{divislemma} that $\Phi_d^*(q)\cdot\Phi_{d-1}^*(q)$ divides the order of $G$. Now, for any integer $e>1$ we have that $\Phi_{ef}^*(p)$ divides $\Phi_e^*(q)$, and hence both $\Phi_{df}^*(p)$ and $\Phi_{(d-1)f}^*(p)$ divide the order of $G$. Hence Theorem~\ref{theoremBP}
 applies. Note that $\F_q$ and $\F_p^f$ are isomorphic as $\F_p$-vector spaces, and so $V$ may also be considered to be a vector space over $\F_p$ of dimension $df$. Thus, when applying Theorem~\ref{theoremBP}
 we consider $G$ to be a subgroup of $\GL_{df}(p)$. 
 
 If $G\leq\GaL_1(q^d)$ as in the extension field case a) of \cite[Theorem~3.1]{bamberg2008overgroups}, 
 then part 2 holds. Suppose that $G$ falls under extension field case b), so that $G\leq \GaL_{df/s}(p^s)$ for some integer $s$, with $1<s<df$ and $s$ dividing both $df$ and $(d-1)f$. This implies that $s$ divides $f$. By assumption, we have that $G\leq \GaL_{d}(q)$, and hence $s=f$. If $f=1$, this gives a contradiction, and hence $G$ is a classical, reducible, imprimitive, symplectic, or nearly simple example, as in part 3. If $f\geq 2$ then the conditions of the extension field case imply that $\Phi_{df}^*(p)$ and $\Phi_{(d-1)f}^*(p)$ divide the order of ${\hat G}=G\cap \GL_d(q)$, and ${\hat G}$ is as in one of the classical, reducible, imprimitive, symplectic, or nearly simple cases of Theorem~\ref{theoremBP},
 treated now as a subgroup of $\GL_d(q)$, and part 3 holds.
\end{proof}

Next we consider the case that part 2 of Lemma~\ref{threecaseslemma} holds.

\begin{lemma}\label{extensionfieldcasea}
 Let $q=p^f$, let $\D=(V,\B)$ be a $2$-$(d,k,\lambda)_q$ design, and let $G$ be the stabiliser of $\B$ inside $\GaL_d(q)$. Moreover, suppose that $G\leq\GaL_1(q^d)$ and both of $\Phi_{df}^*(p)$, $\Phi_{(d-1)f}^*(p)$ are non-trivial and divide the order of $G$. Then $\D$ is not block-transitive.
\end{lemma}

\begin{proof}
 Suppose that $\D$ is block-transitive. Then, since $\Phi_{df}^*(p)$ and $\Phi_{(d-1)f}^*(p)$ are non-trivial and divide the order of $G$, and $G$ is a subgroup of $\GL_{df}(p)$, we may apply \cite[Theorem~3.1]{bamberg2008overgroups}, in which case $G$ falls under the extension field case a). Since $d\geq 6$, we have that $q=p$ and $p^d=2^{11},2^{13},2^{19},3^7$ or $5^7$. Note that, since $d$ is odd in each of these cases, we have that $\D$ is not self-dual and, in particular $\Aut(\D)$ is a subgroup of $\PGaL_d(p)$. Hence $\Aut(\D)$ is the quotient of $G$ by $G\cap Z(\GL_d(p))$, where $Z(\GL_d(p))$ is the centre of $\GL_d(p)$. Thus $|\Aut(\D)|$ divides $d(p^d-1)/(p-1)$. Since $\Aut(\D)$ acts transitively on $\B$, the size of $\B$ must divide the order of $\Aut(\D)$. Now,
 \[
  |\B|=\lambda\frac{(p^d-1)(p^{d-1}-1)}{(p^k-1)(p^{k-1}-1)},
 \] 
 for some integers $\lambda$ and $k$ with $\lambda\geq 1$ and $3\leq k\leq d/2$. Hence the following is an integer:
 \[
  \lambda\frac{|\Aut(\D)|}{|\B|}=\frac{d(p^d-1)}{(p-1)}\frac{(p^k-1)(p^{k-1}-1)}{(p^d-1)(p^{d-1}-1)}=\frac{d(p^k-1)(p^{k-1}-1)}{(p^{d-1}-1)(p-1)}.
 \]
 This is true only if $p^d=2^{11}$ with $k=5$ and $\lambda=1$ or $5$; or $p^d=3^7$ with $k=3$ and $\lambda=1$. By \cite[Theorem~2]{braun2018q}, the derived design of a $2$-$(11,5,1)_2$ design would be a $1$-$(10,4,1)_2$ design. However, by \cite[Lemma~4]{braun2018q}, no $1$-$(10,4,1)_2$ design exists, in particular, such a design would be a spread and $k=4$ does not divide $n=10$. For the case of a $2$-$(7,3,1)_3$ design, the divisibility condition above implies that $\Aut(\D)=\PGaL_1(p^d)$, in which case we may assume that $G=\GaL_1(3^7)$. However, by \cite[Theorem~2~(3)]{miyakawa1995class}, no $2$-$(7,3,1)_3$ design with $G$ acting transitively on $V\setminus \{0\}$ exists. Thus $q=2$, $d=11$ and $k=\lambda=5$. However, Lemma~\ref{jesseslemma} rules out the existence of such a design.
\end{proof}

We now move to consider part 3 of Lemma~\ref{threecaseslemma} holds. The classical cases are excluded by the following, Lemma~\ref{classicalexamples}.

\begin{lemma}[Classical Examples]\label{classicalexamples}
 Let $q=p^f$ and let $\D=(V,\B)$ be a $2$-$(d,k,\lambda)_q$ design with stabiliser ${\hat G}$ of $\B$ inside $\GL_d(q)$. Moreover, suppose that both $\Phi_{df}^*(p)$ and $\Phi_{(d-1)f}^*(p)$ are non-trivial and divide the order of ${\hat G}$. If
${\hat G}$ is a classical example of Theorem~\ref{theoremBP} then $\D$ is not block-transitive.
 \end{lemma}
 
\begin{proof}
 For (a), $\SL_d(q)$ acts transitively on the set of all $k$-spaces of $V$, for any choice of $k$, so that in this case the only subspace design invariant under ${\hat G}$ is the trivial design. 
 
 For case (b) we have that ${\hat G}$ contains $\Sp_d(q)$ as a normal subgroup but does not contain any field automorphisms. Since $d\geq 6$, the only automorphisms of $\Sp_d(q)$ we need to consider are those induced by the centre of $\GL_d(q)$, and hence the order of ${\hat G}$ is at most twice the order of $\Sp_d(q)$ (see \cite[Section~3.5.5]{wilson2009finite}). Hence $|{\hat G}|$ divides $2q^{d^2/4}\prod_{i=1}^{d/2}(q^{2i}-1)$; we claim that this implies that $|{\hat G}|$ is not divisible by $\Phi_{(d-1)f}^*(p)$. Recall that $\Phi^*_e(p)$ is odd for all $e\geq 1$ so that the factor of $2$ is inconsequential. Moreover, $q^{d^2/4}$ is coprime to $\Phi_{(d-1)f}^*(p)$. Thus, we need only be interested in the factors $(q^i-1)$ of $|{\hat G}|$ and these only appear for even $i$. Since $d-1$ is odd, the claim holds and this case does not occur. 
 
 Consider now case (c). By Theorem~\ref{theoremBP},
 we have that $\Phi^*_e(p)$ divides the order of the normaliser of $\SU_d(q^{1/2})$ only when $e$ is odd, but $e$ must be able to take both the values $df$ and $(d-1)f$ here, at least one of which is even.
 
 For case (d), we have that if $d$ is even then ${\hat G}\leq \GO_d^\epsilon(q)$ where $\epsilon=+$ or $-$, and if $d$ is odd then ${\hat G}\leq \GO_d(q)$. Referring to Table~\ref{orthtable}, we see that $\Phi_{(d-1)f}^*(p)$ does not divide the order of $\GO_d^\epsilon(q)$ (it should be noted here though that $\Phi_{df}^*(p)$ divides $q^d+1$, and hence does divide $|{\hat G}|$). Also, $\Phi_{df}^*(p)$ does not divide the order of $\GO_d(q)$. Hence this case does not occur.
\end{proof}
 
\begin{table}[h!]
 \begin{center}
 \begin{tabular}{cc}
  group & order\\
  \hline  
  $\GO_{2m+1}(q)$ & $2q^{m^2}(q^2-1)(q^4-1)\cdots(q^{2m}-1)$\\
  $\GO_{2m}^+(q)$ & $2q^{m(m-1)}(q^2-1)(q^4-1)\cdots(q^{2m-2}-1)(q^m-1)$\\
  $\GO_{2m}^-(q)$ & $2q^{m(m-1)}(q^2-1)(q^4-1)\cdots(q^{2m-2}-1)(q^m+1)$\\  
  \hline
 \end{tabular}
 \caption{Orders of orthogonal groups.}
 \label{orthtable}
 \end{center}
\end{table} 

The reducible cases are excluded by the following, Lemma~\ref{reduciblecase}.

\begin{lemma}[Reducible Examples]\label{reduciblecase}
 Let $q=p^f$ and let $\D=(V,\B)$ be a $2$-$(d,k,\lambda)_q$ design with stabiliser ${\hat G}$ of $\B$ inside $\GL_d(q)$. Moreover, suppose that both $\Phi_{df}^*(p)$ and $\Phi_{(d-1)f}^*(p)$ are non-trivial and divide the order of ${\hat G}$. If
${\hat G}$ is a reducible example of Theorem~\ref{theoremBP} then $\D$ is not block-transitive.
 \end{lemma}

\begin{proof} 
 In this case we require ${\hat G}  \leq H\cong q^{m(d-m)}\cdot (\GL_m(q)\times\GL_{d-m}(q))$ for some $m$ such that $0<m<d$. By definition, $\Phi_{df}^*(p)$ is coprime to each factor $q^i-1$ dividing $|H|$, that is, for each $i\leq {\rm{max}} \{m,d-m\}<d$. Since $\Phi_{df}^*(p)$ divides $q^d-1$, it follows that $\Phi_{df}^*(p)$ is also coprime to $q^{m(d-m)}$, and hence does not divide the order of ${\hat G}$.
\end{proof}

\begin{lemma}[Imprimitive Examples]\label{imprimitivecase}
 Let $q=p^f$ and let $\D=(V,\B)$ be a $2$-$(d,k,\lambda)_q$ design with stabiliser ${\hat G}$ of $\B$ inside $\GL_d(q)$. Moreover, suppose that both $\Phi_{df}^*(p)$ and $\Phi_{(d-1)f}^*(p)$ are non-trivial and divide the order of ${\hat G}$. If
${\hat G}$ is an imprimitive example of Theorem~\ref{theoremBP} then $\D$ is not block-transitive.
\end{lemma}

\begin{proof}
 In this case the values of $q$, $d$ and $e$ for which $\Phi_{ef}^*(p)$ divide ${\hat G}$ are given in the relevant table of \cite{bamberg2008overgroups}. Since we require that $e$ take both values $d$ and $d-1$, and all given values of $e$ are even, these cases do not occur.
\end{proof}

\begin{lemma}[Symplectic Type Examples]\label{symplecticcase}
Let $q=p^f$ and let $\D=(V,\B)$ be a $2$-$(d,k,\lambda)_q$ design with stabiliser ${\hat G}$ of $\B$ inside $\GL_d(q)$. Moreover, suppose that both $\Phi_{df}^*(p)$ and $\Phi_{(d-1)f}^*(p)$ are non-trivial and divide the order of ${\hat G}$. If
${\hat G}$ is a symplectic type example of Theorem~\ref{theoremBP} then $\D$ is not block-transitive.
\end{lemma}

\begin{proof}
 In this case the values of $q$, $d$ and $e$ for which $\Phi_{ef}^*(p)$ divide ${\hat G}$ are given in \cite[Theorem~3.1]{bamberg2008overgroups}. Since we require that $e$ take both values $d$ and $d-1$, and in each case there is only a single value for $e$, these cases do not occur.
\end{proof}

\begin{lemma}[Nearly Simple Examples]\label{nearlysimplecases}
Let $q=p^f$ and let $\D=(V,\B)$ be a $2$-$(d,k,\lambda)_q$ design with stabiliser ${\hat G}$ of $\B$ inside $\GL_d(q)$. Moreover, suppose that both $\Phi_{df}^*(p)$ and $\Phi_{(d-1)f}^*(p)$ are non-trivial and divide the order of ${\hat G}$. If
${\hat G}$ is a nearly simple example of Theorem~\ref{theoremBP} then $\D$ is not block-transitive.
\end{lemma}

\begin{proof}
 We consider for ${\hat G}$ the groups as in the nearly simple case of Theorem~\ref{theoremBP}
  (see \cite[pages 2507--2508]{bamberg2008overgroups} for the details of each case and the tables we refer to in this proof). First, suppose ${\hat G}$ is a permutation module example under the alternating group sub-case, so that $q=2,3$ or $5$. If $q=2$ then, since $\Phi_{d}^*(2)$ must divide the order of ${\hat G}$, we have that $d=4,10,12$ or $18$. However, we also have that $\Phi_{d-1}^*(2)$ must divide the order of ${\hat G}$, from which it follows that $d=5,11,13$ or $19$. This gives a contradiction. If $q=3$ then, similarly, we have that $d=4$ or $6$, and at the same time $d=5$ or $7$, a contradiction. If $q=5$ then we again reach a contradiction, requiring this time that $d$ is equal to $6$ and $7$ at the same time.
 
 Consider now the other examples of the alternating group sub-case and the corresponding table of \cite{bamberg2008overgroups}. Recall that $d$ is at least $6$, and notice that every entry in the table with $d\geq 6$ also has $d\neq e$. Hence $\Phi_{d}^*(q)$ does not divide the order of ${\hat G}$, contradicting Lemma~\ref{divislemma}.
 
 In the sporadic simple group sub-case, and the corresponding table of \cite{bamberg2008overgroups}, each pair $(G',d)$ that does occur appears in precisely one entry of the table. Hence there is a unique value of $e$ in each case. Thus it is impossible for both $\Phi_{d}^*(q)$ and $\Phi_{d-1}^*(q)$ to divide $|{\hat G}|$.
 
 Consider the table of \cite{bamberg2008overgroups} corresponding to the cross-characteristic sub-case. The only group appearing in the table having multiple values of $e$ for a given $q$ is $\Sp_6(2)$ with $d=7$ and $q=3$, in which case $e$ takes the values $4$ and $6$, neither of which are equal to $d$.
 
 Finally, in the table corresponding to the natural-characteristic sub-case, each group has only one value for $e$ with a given $d$. This rules out all of the groups in the nearly simple cases.
\end{proof}

We are now able to provide the proof of Theorem~\ref{maintheorem}.

\begin{proof}[Proof of Theorem~\ref{maintheorem}]
 Suppose $\D=(V,\B)$ is a non-trivial block-transitive $t$-$(d,k,\lambda)_q$ design, where $V=\F_q^d$, and let $q=p^f$, where $p$ is prime. By Lemma~\ref{twodesigncor}, we may assume that $t=2$, and hence that $k\geq 3$. By Lemma~\ref{dualdesignlemma}, we may assume that $k\leq d/2$, so that $d\geq 6$. Let $G$ be the setwise stabiliser of $\B$ inside $\GaL_d(q)$. We may then apply Lemma~\ref{threecaseslemma}. For part 1 of Lemma~\ref{threecaseslemma}, Zsigmondy's theorem \cite{zsigmondy1892theorie} states that if $\Phi_e^*(p)=1$ then $p^e=2^6$. However, Lemmas~\ref{nodesignnequals6} and~\ref{nminus1is6} rule out the existence of any $2$-$(d,k,\lambda)_q$ design with $q^d$ or $q^{d-1}$ equal to $2^6$. Lemma~\ref{extensionfieldcasea} rules out the occurrence of part 2 of Lemma~\ref{threecaseslemma}. Part 3 of Lemma~\ref{threecaseslemma} is ruled out by Lemma~\ref{classicalexamples}, Lemma~\ref{reduciblecase}, Lemma~\ref{imprimitivecase}, Lemma~\ref{symplecticcase}, Lemma~\ref{nearlysimplecases}. This completes the proof of Theorem~\ref{maintheorem}.
\end{proof}


\end{document}